\documentclass[twoside, 11pt]{article}

\usepackage{graphicx, subfigure}
\usepackage[top=1in,bottom=1in,left=1in,right=1in,footnotesep=0.3in]{geometry}
\usepackage[sort&compress]{natbib} 
\usepackage{amsfonts, amsmath, amssymb, amsthm, constants, bbm}
\usepackage{MnSymbol}
\usepackage{mathtools}
\usepackage{enumitem}
\usepackage{multirow}
\usepackage{booktabs}

\usepackage{color}
\definecolor{darkred}{RGB}{100,0,0}
\definecolor{darkgreen}{RGB}{0,100,0}
\definecolor{darkblue}{RGB}{0,0,150}

\usepackage{hyperref}
\hypersetup{colorlinks=true, linkcolor=darkred, citecolor=darkgreen, urlcolor=darkblue}
\usepackage{url}

\def\d{{\rm d}}

\newtheorem{theorem}{Theorem}[section]
\newtheorem{prp}[theorem]{Proposition}
\newtheorem{lem}[theorem]{Lemma}
\newtheorem{cor}[theorem]{Corollary}

\theoremstyle{remark}

\newtheorem{asp}{Assumption}

\def\beq{\begin{equation}} 
\def\eeq{\end{equation}}
\def\beqn{\begin{eqnarray*}}
\def\eeqn{\end{eqnarray*}}
\def\Bitem{\begin{itemize}\setlength{\itemsep}{.2in}}
\def\bitem{\begin{itemize}\setlength{\itemsep}{.05in}}
\def\eitem{\end{itemize}}
\def\Benum{\begin{enumerate}\setlength{\itemsep}{.2in}}
\def\benum{\begin{enumerate}\setlength{\itemsep}{.05in}}
\def\eenum{\end{enumerate}}
\def\bmult{\begin{multline*}}
\def\emult{\end{multline*}}
\def\bcenter{\begin{center}}
\def\ecenter{\end{center}}
\def\bframe{\begin{frame}}
\def\eframe{\end{frame}}

\newcommand{\thmref}[1]{Theorem~\ref{thm:#1}}
\newcommand{\prpref}[1]{Proposition~\ref{prp:#1}}

\newcommand{\lemref}[1]{Lemma~\ref{lem:#1}}
\newcommand{\secref}[1]{Section~\ref{sec:#1}}
\newcommand{\figref}[1]{Figure~\ref{fig:#1}}
\newcommand{\tabref}[1]{Table~\ref{tab:#1}}

\DeclareMathOperator*{\argmax}{arg\, max}
\DeclareMathOperator*{\argmin}{arg\, min}

\def\cG{\mathcal{G}}

\def\cN{\mathcal{N}}

\def\bbI{\mathbb{I}}

\def\bbR{\mathbb{R}}

\newcommand{\E}{\operatorname{\mathbb{E}}}
\renewcommand{\P}{\operatorname{\mathbb{P}}}
\newcommand{\Var}{\operatorname{Var}}

\def\eps{\varepsilon}

\def\1{\mathbbm{1}}
\newcommand{\IND}[1]{\bbI{\{ #1 \}}}

\definecolor{purple}{rgb}{0.4,.1,.9}

\newcommand\blfootnote[1]{%
  \begingroup
  \renewcommand\thefootnote{}\footnote{#1}%
  \addtocounter{footnote}{-1}%
  \endgroup
}

\begin{document}
\thispagestyle{empty}

\title{Template Matching with Ranks}
\author{Ery Arias-Castro \and Lin Zheng}
\date{}
\maketitle

\blfootnote{The authors are with the Department of Mathematics, University of California, San Diego, USA.  Contact information is available \href{http://math.ucsd.edu/\~eariasca}{here} and \href{http://www.math.ucsd.edu/people/graduate-students/}{here}.}

\begin{abstract}
We consider the problem of matching a template to a noisy signal. Motivated by some recent proposals in the signal processing literature, we suggest a rank-based method and study its asymptotic properties using some well-established techniques in empirical process theory combined with H\'ajek's projection method. The resulting estimator of the shift is shown to achieve a parametric rate of convergence and to be asymptotically normal. Some numerical simulations corroborate these findings.
\end{abstract}

\noindent
{\small {\bf Keywords:}
matched filter, template matching, scan statistics, Spearman rank correlation, empirical processes, minimax optimality}


\section{Introduction} 
\label{sec:intro}
Template matching is the process of matching a clean and noiseless template to an observed, typically noisy signal. This topic is closely related to the problem of matching two or more noisy signals, sometimes referred to as `aligning' or `registering' the signals, and to methodology in spatial statistics falling under the umbrella name of `scan statistic'. When the template has a point of discontinuity, matching a template to the signal can be interpreted as detecting the location of the discontinuity, a more specialized task more broadly referred to as `change-point detection' in statistics.
For pointers to the vast literature on these related topics, we refer the reader to the introduction of another recent paper of ours on the topic \citep{arias2020template}.

While in our previous work we proposed and analyzed M-estimators for template matching, in the present paper we consider R-estimators instead. The former includes what is perhaps the most widely used method which consists in maximizing the Pearson correlation of the signal with a shift of the template; see \eqref{pearson} below. 
We focus here on the rank variant of this approach, which is an example of the latter category and consists, instead, in replacing the signal with the corresponding ranks before maximizing the correlation over shifts of the template; see \eqref{spearman}. 

Rank-based methods are, of course, classical in statistics \citep{sidak1999theory, hettmansperger2010robust, lehmann2006nonparametrics, gibbons2014nonparametric}. 
The theory of rank tests is particularly well-developed, featuring some prominent methods such as the Wilcoxon/Mann--Whitney and Kruskal--Wallis two- and multi-sample tests, and essentially all (other) distribution-free tests for goodness-of-fit such as the Kolmogorov--Smirnov test, and more closely related to our topic here, the Spearman and Kendall rank correlation tests for independence. 
The theory of rank estimators, sometimes called R-estimators, is also well developed, although perhaps not as well-known \citep{hettmansperger2010robust}. The most famous example may be the Hodges--Lehmann estimator, which is derived from the Wilcoxon signed-rank test.
In multiple linear regression, the asymptotic linearity and resulting normality of certain R-estimators is established in a number of publications \citep{jureckova1971nonparametric, heiler1988asymptotic, koul1993asymptotics, giraitis1996asymptotic, draper1988rank}.
Closer to our setting, some papers consider the use of ranks for the detection and/or localization of one or multiple change-points \citep{darkhovskh1976nonparametric, gombay1998rank, huvskova1997limit, lung2011homogeneity, gerstenberger2018robust, wang2020rank, arias2018distribution}.

In the signal and image processing literature per se, where a lot of the work on template matching resides, rank-based methods have been attracting some attention in recent years.  
\cite{kordelas2009robust} propose a rank variant of the well-known feature extractor SIFT, while \cite{xiong2019rank} propose a rank-based local self-similarity feature descriptor for use in synthetic-aperture radar (SAR) imaging. 
\cite{ayinde2002face} present a face recognition approach using the rank correlation of Gabor filtered images, while \cite{galea2016face} apply the Spearman rank correlation for template matching in face photo-sketch recognition. 
\cite{kong2008increasing} construct a matched-filter object detection algorithm based on the Spearman rank correlation to detect Ca$^{2+}$ sparks in biochemical applications.  
A number of papers use ranks to align images, a task also known as `stereo matching' \citep{banks2001reliability, banks1999constraint, chen2012multi, geng2012adaptive} --- a problem we will not address here but which also has a sizable literature in statistics; we provide some pointers to the literature in our recent paper \citep{arias2020template}.

\subsection{Model and methods}
\label{sec:model}

We consider a standard model for template matching, where we observe a shift of the template with additive noise, 
\beq\label{model}
Y_i = f(x_i - \theta^*) + Z_i, \quad i = 1, \dots, n,
\eeq
where $x_i := i/n$ denote the {\em design points}, $f:[0,1] \to \bbR$ is a known 1-periodic function referred to as the {\em template}, and $\theta^* \in \bbR$ is the unknown {\em shift} of interest.
The {\em noise} or measurement error variables $Z_1, \dots, Z_n$ are assumed iid with density $\phi$ and distribution function denoted $\Phi$.

\begin{asp}
\label{asp:template}
We assume everywhere that $f$ is 1-periodic, and in fact exactly so in the sense that $f(\cdot-\theta) \ne f(\cdot-\theta^*)$ on a set of positive measure whenever $\theta \ne \theta^*\! \mod 1$. We also assume that $f$ is Lipschitz continuous.
\end{asp}

\begin{asp}
\label{asp:noise}
We assume everywhere that $\phi$ is even, so that the noise is symmetric about 0.
\end{asp}

Our goal is, in signal processing terminology, to match the template $f$ to the signal $Y$, which in statistical terms consists in the estimation of the shift $\theta^*$. One of the most popular ways to do so is via maximization of the Pearson correlation, leading to the estimator
\begin{equation}
\label{pearson}
\hat\theta := \argmax_{\theta} \sum_{i=1}^n Y_i\,f(x_i - \theta).
\end{equation}
In the present context, this is equivalent to the least squares estimator in that the same estimator is also the solution to the following least squares problem 
\begin{equation}
\label{ls}
\hat\theta = \argmin_{\theta} \sum_{i=1}^n (Y_i - f(x_i - \theta))^2.
\end{equation}
Note that this corresponds to the maximum likelihood estimator when the noise distribution is Gaussian. Of course, other loss functions can be used, some of them leading to methods that are robust to noise distributions with heavy tails or to the presence of gross errors (outliers) in the signal. Our recent paper \citep{arias2020template} studies these so-called {\em M-estimators} in great detail. 

In the present paper we consider, instead, estimators based on ranks. These go by the name of {\em R-estimators} in the statistics literature. The most direct route to such an estimator is to replace the response values with their ranks, yielding
\begin{equation}
\label{spearman}
\hat\theta_{\rm rank} := \argmax_{\theta} \sum_{i=1}^n R_i\,f(x_i - \theta),
\end{equation}
where $R_i$ denotes the rank of $Y_i$ in $\{Y_1, \dots, Y_n\}$ in increasing order. Doing so is sometimes called the `rank transformation' in the signal processing literature.
Note that this is similar in spirit to replacing the Pearson correlation in \eqref{pearson} with the Spearman rank correlation, except that we do not rank the values of the template itself. We find that there is no real reason to want to replace the template values with the corresponding ranks as the template is assumed to be free of noise.

\subsection{Content}
The rest of the paper is devoted to studying the asymptotic ($n\to\infty$) properties of the estimator we propose in \eqref{spearman}, which we will refer to as the R-estimator. 
In \secref{theory}, we derive some basic results describing the asymptotic behavior of this estimator. In more detail, in \secref{consistency} we discuss the consistency of this estimator, which we are not fully able to establish but is clearly supported by computer simulations; in \secref{rate} we derive a rate of convergence for our R-estimator, which happens to be parametric and also minimax optimal; and in \secref{distribution} we derive a normal limit distribution of the same estimator, as well as its asymptotic relative efficiency with respect to maximum likelihood estimator under Gaussian noise. 
\secref{numerics} summarizes the result of some numerical experiments we performed to probe our asymptotic theory in finite samples. 
\secref{discussion} provides a discussion of possible extensions.
The mathematical proofs are gathered in \secref{proofs}.

\section{Theoretical properties}
\label{sec:theory}

In this section we study the asymptotic properties of the estimator defined in \eqref{spearman}. We first study its consistency in \secref{consistency}; bound its rate of convergence in \secref{rate}; and derive its asymptotic (normal) limit distribution in \secref{distribution}.

\subsection{Consistency}
\label{sec:consistency}

The estimator in \eqref{spearman} can be equivalently defined via
\begin{align}
\label{R-estimator}
\hat\theta_n := \argmax_\theta \widehat{M}_n(\theta), &&
\widehat{M}_n(\theta) := \frac1n \sum_{i=1}^n \frac{R_i}{n} f(x_i-\theta),
\end{align}
where we have added the subscript $n$ to emphasize that the estimator is being computed on a sample of size $n$.
To understand the large-sample behavior of this estimator, we need to understand that of $\widehat{M}_n$ as a function of $\theta$, and this leads directly to empirical process theory.

\begin{lem} \label{lem:uniform_convergence}
In probability,
\begin{align*}
\sup_{\theta \in \bbR} \big|\widehat M_n(\theta) - M(\theta)\big|
\xrightarrow{n\to\infty} 0,
\end{align*}
where
\begin{align}
\label{M}
M(\theta) := \int_0^1 \int_0^1 \Phi_2(f(x_0)-f(x)) f(x_0 + \theta^*-\theta) \d x \d x_0,
\end{align}
with
$\Phi_2(t) := \int_{-\infty}^\infty \Phi(t + z) \phi(z) \d z.$
\end{lem}

With the uniform convergence of $\widehat M_n$ to $M$ established in \lemref{uniform_convergence}, and with the fact that $M$ is continuous and 1-periodic, it suffices that $\theta^*$ be the unique maximizer of $M$ for the estimator $\hat\theta_n$ defined in \eqref{R-estimator} to converge to $\theta^*$ in probability, that is, to be consistent \cite[Th 5.7]{van2000asymptotic}.
We are able, under an additional mild assumption on the noise distribution, to prove that $\theta^*$ is a local maximizer of $M$. 

\begin{prp}
\label{prp:local}
$M$ is twice differentiable, with $M'(\theta^*) = 0$, and if $\phi$ is positive everywhere, $M''(\theta^*) < 0$, in which case $\theta^*$ is a local maximizer of $M$.
\end{prp}

Unfortunately, we are not able to verify that $\theta^*$ is indeed a global maximizer. However, since it is borne out by our numerical experiments (\secref{numerics}), we conjecture this is true, possibly under some additional (reasonable) conditions on the model. In the meantime, we formulate this as an assumption, which henceforth forms part of our basic assumptions.

\begin{asp}
\label{asp:global}
$\theta^*$ is the unique maximum point of $M$ defined in \eqref{M}.
\end{asp} 

\begin{theorem}
\label{thm:consistent}
Under the basic assumptions, $\hat\theta_n$ converges in probability to $\theta^*$ as $n\to\infty$.
\end{theorem}

\subsection{Rate of convergence and minimaxity}
\label{sec:rate}
Besides consistency, we derive the estimator's rate of convergence in this section. The rate turns out to be parametric, i.e., the convergence of the estimator to the true value of the parameter is in $O(\sqrt{n})$. 

\begin{theorem}
\label{thm:rate}
Under the basic assumptions, $\hat\theta_n$ is $\sqrt{n}$-consistent.
\end{theorem}

The parametric rate of $\sqrt{n}$ happens to be minimax optimal in the present setting. This is established in \citep[Cor~3.9]{arias2020template} in the context of a random design corresponding to the situation where the design points, instead of being the grid points spanning the unit interval, are generated as an iid sample from the uniform distribution on the unit interval. But similar arguments carry over. We omit details and refer the reader to the discussion in \cite[Sec~6.1]{arias2020template}.

\begin{cor}
\label{cor:minimax}
The R-estimator achieves the minimax rate of convergence.
\end{cor}

\subsection{Limit distribution and asymptotic relative efficiency}
\label{sec:distribution}
In addition to obtaining a rate of convergence, we are also able to derive the asymptotic distribution, which happens to be normal. 

\begin{theorem}
\label{thm:normal}
Under the basic assumptions, strengthened with the assumption that $f$ is continuously differentiable, $\sqrt{n}(\hat\theta_n - \theta^*)$ converges weakly to the centered normal distribution with variance $\gamma^2/M''(\theta^*)^2$, where
\begin{align*}
\gamma^2 
:= \int_0^1 \int_{-\infty}^\infty \bigg[\int_0^1 (f'(x) - f'(x_0))\, \Xi(f(x)-f(x_0), z) \d x\bigg]^2 \phi(z) \d z \d x_0,
\end{align*}
with $\Xi(w, z) := \Phi(z+w) - \Phi_2(w)$, and
\begin{align*}
M''(\theta^*)
= - \int_0^1 \int_0^1  f'(x)^2 \phi_2(f(x) - f(x_0)) \d x d x_0,
\end{align*}
with 
$\phi_2(t) := \Phi_2'(t) = \int_{-\infty}^\infty \phi(z + t) \phi(z) \d z.$
\end{theorem}

Now that the R-estimator is known to be asymptotically normal with an explicit expression for the asymptotic variance (after standardization), we can consider its (Pitman) efficiency relative to the more popular estimator based on maximizing the Pearson correlation \eqref{pearson} (which coincides with the MLE when the noise is Gaussian). 
Indeed, this estimator (denoted $\tilde\theta_n$ now) was studied in \cite[Sec~3]{arias2020template} in the setting of a random design. Adapting the arguments there, which are very similar to (and in fact simpler than) those used here, we find that $\sqrt{n}(\tilde\theta_n - \theta^*)$ is asymptotically normal with mean zero and variance $\sigma^2/\int_0^1 f'(x)^2 \d x$, where $\sigma^2$ is the noise variance.
With \thmref{normal}, we are thus able to conclude the following.

\begin{cor}
Suppose that the noise has finite variance $\sigma^2$. 
Then the asymptotic efficiency of the R-estimator \eqref{spearman} relative to the standard estimator \eqref{pearson} is given by 
\begin{align*}
\frac{\sigma^2/\int_0^1 f'(x)^2 \d x}{\gamma^2/M''(\theta^*)^2}.
\end{align*}
\end{cor}

The result, in fact, continues to hold even when the noise has infinite variance, and in that case the asymptotic relative efficiency of the R-estimator relative to the more common estimator is infinite.

\section{Numerical experiments}
\label{sec:numerics}
We performed some simple numerical experiments to probe the asymptotic theory developed in the previous sections. We note that the implementation of the R-estimator \eqref{spearman} is completely straightforward, as it only requires replacing the observations with their ranks before the usual template matching by maximization of the correlation over shifts as in \eqref{pearson}, which is typically implemented by a fast Fourier transform. This ease of computation is in contrast with rank methods for, say, linear regression which are computationally demanding and require dedicated algorithms and implementations --- difficulties that may explain the very limited adoption of such methods in practice. 

In our experiments, we considered three noise distributions: Gaussian distribution, Student t-distribution with 3 degrees of freedom, and Cauchy distribution. We took $\theta^*=0$ throughout, which is really without loss of generality since the two methods we compared --- the standard method based on maximizing the Pearson correlation and the rank-based method that we study --- are translation equivariant.
We chose to work with the following three filters:
\beq
\text{Template $A$: } \;\;f(x)=
\begin{cases}
4x-1& 0.25\leq x<0.5,\\
3-4x& 0.5 \leq x<0.75,\\
0& \text{otherwise,}
\end{cases}
\eeq
\beq
\text{Template $B$: } \;\;f(x)=
\begin{cases}
10x-2& 0.2\leq x<0.3,\\
4-10x& 0.3 \leq x<0.4,\\
10x-6& 0.6\leq x<0.7,\\
8-10x& 0.7\leq x<0.8,\\
0& \text{otherwise}
\end{cases}
\eeq
and
\beq
 \text{Template $C$: }\;\;f(x) = \max\{0, (1-(4x-2)^2)^3\}.
\eeq
All are Lipschitz, with Template $C$ being even smoother. See \figref{Fig-1} for an illustration.

\begin{figure}[htbp]
\centering 
\subfigure[Template $A$]{
\label{fig:Fig-1-1}
\includegraphics[width=0.32\textwidth]{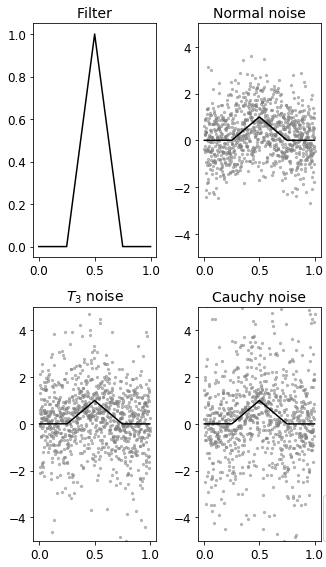}}
\subfigure[Template $B$]{
\label{fig:Fig-1-2}
\includegraphics[width=0.32\textwidth]{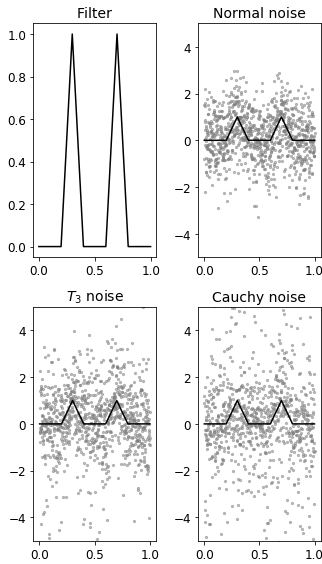}}
\subfigure[Template $C$]{
\label{fig:Fig-1-3}
\includegraphics[width=0.32\textwidth]{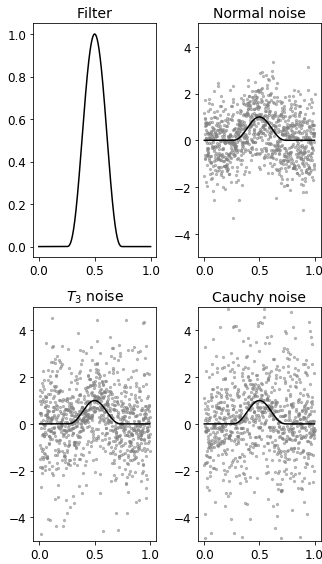}}
\caption{Templates and noisy signals. Although the sample size is $n = 10000$, for the sake of clarity, we only include $1000$ points and limit the range of the y-axis to $[-5,5]$.}
\label{fig:Fig-1}
\end{figure}

We set the sample size at $n = 10000$. Each setting, defined by a choice of filter and of noise distribution, was repeated $500$ times.  
Box plots of estimation error $|\hat{\theta}_n-\theta^*|$ are presented in Figures~\ref{Fig-2}, \ref{Fig-3} and~\ref{Fig-4}, while the distribution of $\sqrt{n}(\hat{\theta}_n-\theta^*)$ is depicted in Figures~\ref{Fig-5}, \ref{Fig-6} and~\ref{Fig-7} via histograms. The results are congruent with what the theory predicts: The rank-based method is slightly inferior to the standard method when the noise is Gaussian (exactly when the standard method coincides with the MLE), while it is superior when the noise distribution has heavier tails; and the histograms overlay nicely with the predicted asymptotic distribution. 
The asymptotic relative efficiency of R-estimator relative to the method based on maximizing the Pearson correlation is displayed in \tabref{ARE}.


\begin{table}[htbp]
\centering
\begin{tabular}{ cccc } 
\toprule  
\multirow{2}{*}{Template} & \multicolumn{3}{c}{Noise}\\
&Normal & Student $t_3$  & Cauchy\\
\midrule
Template $A$ & 0.949 & 2.008  & $\infty$ \\ 
Template $B$ & 0.940 & 1.992  & $\infty$ \\ 
Template $C$ & 0.948 & 2.008  & $\infty$ \\ 
\bottomrule
\end{tabular}
\caption{Asymptotic relative efficiency of the R-estimator \eqref{spearman} to the more common estimator \eqref{pearson}. Note that the latter is asymptotically best in the setting of Gaussian noise as it then coincides with the maximum likelihood estimator for a `smooth' model.}
\label{tab:ARE}
\end{table}

%

\begin{figure}[htbp]
\centering 
\includegraphics[width=0.9\textwidth]{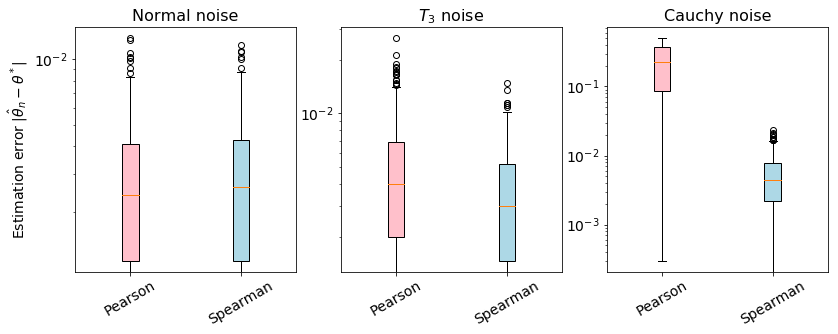}
\caption{Box plot of estimation error $|\hat{\theta}_n-\theta^*|$ for Template $A$}
\label{Fig-2}
\end{figure}

\begin{figure}[htbp]
\centering 
\includegraphics[width=0.8\textwidth]{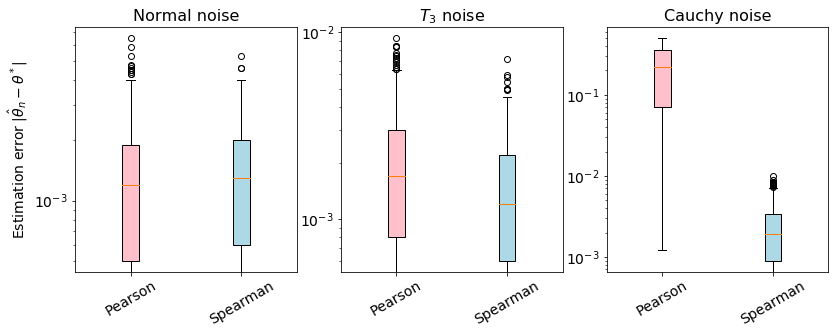}
\caption{Box plot of estimation error $|\hat{\theta}_n-\theta^*|$ for Template $B$}
\label{Fig-3}
\end{figure}

\begin{figure}[htbp]
\centering 
\includegraphics[width=0.8\textwidth]{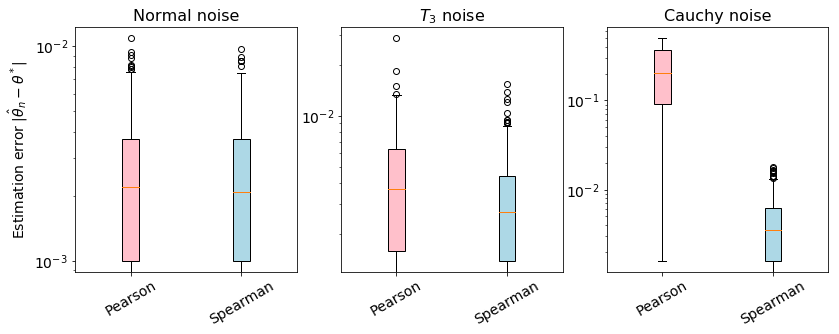}
\caption{Box plot of estimation error $|\hat{\theta}_n-\theta^*|$ for Template $C$}
\label{Fig-4}
\end{figure}

\begin{figure}[htbp]
\centering 
\subfigure[Normal noise]{
\label{Fig-5-1}
\includegraphics[width=0.32\textwidth]{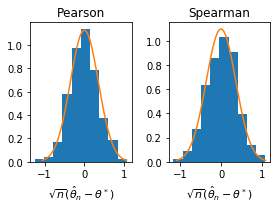}}
\subfigure[$T_3$ noise]{
\label{Fig-5-2}
\includegraphics[width=0.32\textwidth]{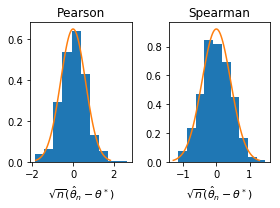}}
\subfigure[Cauchy noise]{
\label{Fig-5-3}
\includegraphics[width=0.32\textwidth]{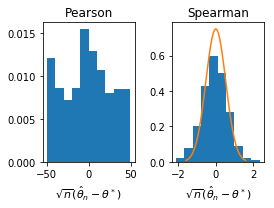}}
\caption{Distribution under Template $A$. The histogram presents the distribution of $\sqrt{n}(\hat{\theta}_n-\theta^*)$. The orange bell-shaped curve is the density of normal distribution predicted by the theory.}
\label{Fig-5}
\end{figure}

\begin{figure}[htbp]
\centering 
\subfigure[Normal noise]{
\label{Fig-6-1}
\includegraphics[width=0.32\textwidth]{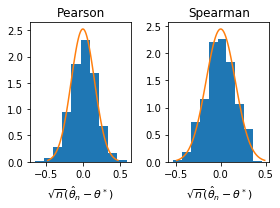}}
\subfigure[$T_3$ noise]{
\label{Fig-6-2}
\includegraphics[width=0.32\textwidth]{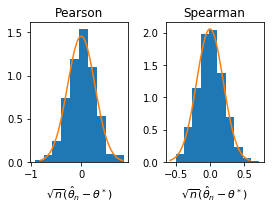}}
\subfigure[Cauchy noise]{
\label{Fig-6-3}
\includegraphics[width=0.32\textwidth]{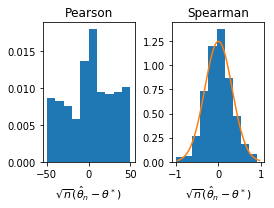}}
\caption{Distribution under Template $B$. The histogram presents the distribution of $\sqrt{n}(\hat{\theta}_n-\theta^*)$. The orange bell-shaped curve is the density of normal distribution predicted by the theory.}
\label{Fig-6}
\end{figure}

\begin{figure}[htbp]
\centering 
\subfigure[Normal noise]{
\label{Fig-7-1}
\includegraphics[width=0.32\textwidth]{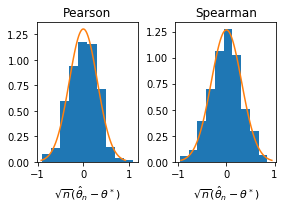}}
\subfigure[$T_3$ noise]{
\label{Fig-7-2}
\includegraphics[width=0.32\textwidth]{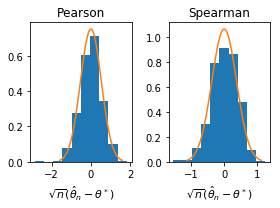}}
\subfigure[Cauchy noise]{
\label{Fig-7-3}
\includegraphics[width=0.32\textwidth]{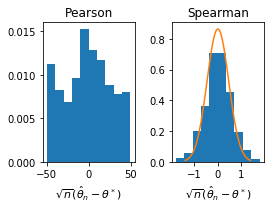}}
\caption{Distribution under Template $C$. The histogram presents the distribution of $\sqrt{n}(\hat{\theta}_n-\theta^*)$. The orange bell-shaped curve is the density of normal distribution predicted by the theory.}
\label{Fig-7}
\end{figure}

\section{Discussion}
\label{sec:discussion}

Our main goal in this paper was to show that a standard rank-based approach to template matching is viable and amenable to study using well-established techniques in mathematical statistics --- some basic results in empirical process theory and the projection method of H\'ajek. This provides some theoretical foundation for related approaches proposed in recent years in the signal processing literature.
We chose to keep the exposition contained and, in particular, have focused on the `smooth setting' of a Lipschitz template. We leave the equally important case of a discontinuous template for future work (likely by others). Based on our previous work \citep{arias2020template} --- where we did study this case in detail --- and on our work here --- in particular H\'ajek's projection technique used in the proof of \thmref{normal} --- we do believe that the study of this case is within the reach of similarly standard tools.

Some of the other extensions discussed in our previous work are also relevant here. In particular, while we focused on shifts in the context of 1D signals, other settings are possible, including shifts in 2D or 3D signals (i.e., images), as well as other transformations. We omit details and simply affirm that such extensions are also amenable to a similar mathematical analysis.

In signal processing, rank-based methods seem more prominently represented in the literature on signal registration. We are confident that the study of such methods is well within the range of established techniques in mathematical statistics. However, the situation becomes substantially more complex as 1) the setting is semi-parametric, and 2) some smoothing seems to be required to achieve good performance --- as transpires from the statistics literature on the topic as mentioned in our previous work \citep{arias2020template}. We leave a further exploration of rank-based methods for registration to future endeavors.

\section{Proofs}
\label{sec:proofs}

\subsection{Preliminaries}
The following is an extension of the celebrated Glivenko--Cantelli theorem.
\begin{lem}
\label{lem:GC}
Suppose that $\{Y_{i,n} : i \in [n], n \ge 1\}$ are independent random variables with uniformly tight and equicontinuous distributions $\{F_{i,n} : i \in [n], n \ge 1\}$.
Define 
\begin{align*}
\hat F_{[n]}(y) := \frac1n \sum_{i=1}^n \IND{Y_{i,n} \le y}, && F_{[n]} := \E[\hat F_{[n]}] = \frac1n \sum_{i=1}^n F_{i,n}.
\end{align*}
Then the following uniform convergence holds in probability
\begin{align*}
\sup_t |\hat F_{[n]}(t) - F_{[n]}(t)| 
\xrightarrow{n\to\infty} 0.
\end{align*}
\end{lem}

\begin{proof}
The proof arguments are very close to those supporting the more classical situation in which the random variables have the same distribution \cite[Th 19.1]{van2000asymptotic}. We provide a proof for completeness only.

Our assumptions mean 1) that for each $\eps>0$ there is $K$ such that $\inf_i F_{i,n}(K) \ge 1-\eps$ and $\sup_i F_{i,n}(-K) \le \eps$; and 2) that $\omega(\delta) := \sup_i \sup_t (F_{i,n}(t+\delta)-F_{i,n}(t))$ is continuous at 0; we then speak of $\omega$ as the modulus of continuity of the family $\{F_{i,n}\}$.
Fix $\eps > 0$ and let $K$ be defined as above.
Also, let $\delta$ be such that $\omega(\delta) \le \eps$ and set $-K = t_1 < t_2 < \dots < t_{m-1} < t_m = K$ be such that $t_{j+1} -t_j \le \delta$ for all $j$.
We then have, regardless of $n$, that $F_{[n]}(K) \ge 1-\eps$ and $F_{[n]}(-K) \le \eps$; and that $\omega$ is a modulus of continuity for $F_{[n]}$, meaning that $\sup_t (F_{[n]}(t+\delta)-F_{[n]}(t)) \le \omega(\delta)$.
\begin{itemize}
\item For $t \le t_1$, we have
\begin{align*}
|\hat F_{[n]}(t) - F_{[n]}(t)| 
\le \hat F_{[n]}(t) + F_{[n]}(t)
\le \hat F_{[n]}(t_1) + \eps.
\end{align*}
\item For $t \ge t_m$, we have
\begin{align*}
|\hat F_{[n]}(t) - F_{[n]}(t)| 
\le (1-\hat F_{[n]}(t)) + (1-F_{[n]}(t))
\le (1-\hat F_{[n]}(t_m)) + \eps.
\end{align*}
\item For $-K < t < K$, if $j$ is such that $t_j < t \le t_{j+1}$,
\begin{align*}
\hat F_{[n]}(t) - F_{[n]}(t)
\le \hat F_{[n]}(t_{j+1}) - F_{[n]}(t_j)
\le \hat F_{[n]}(t_{j+1}) - F_{[n]}(t_{j+1}) + \eps,
\end{align*}
and
\begin{align*}
\hat F_{[n]}(t) - F_{[n]}(t)
\ge \hat F_{[n]}(t_{j}) - F_{[n]}(t_{j+1})
\le \hat F_{[n]}(t_{j}) - F_{[n]}(t_{j}) - \eps.
\end{align*}
\end{itemize}
By Chebyshev's inequality, $\hat F_{[n]}(t) - F_{[n]}(t) \to 0$ in probability as $n\to\infty$ for every fixed $t \in \bbR$. In particular,  \smash{$\max_j |\hat F_{[n]}(t_{j})- F_{[n]}(t_{j})| \to 0$} in probability, and under the event that this maximum is bounded by $\eps$, we have
$|\hat F_{[n]}(t) - F_{[n]}(t)| \le 2\eps$. 
Since $\eps > 0$ is arbitrary, the proof is complete.
\end{proof}

The following is a simple result on functions defined as the linear combination of uniformly equicontinuous functions with random coefficients.
\begin{lem}
\label{lem:simple_EP}
Suppose that $\{B_{i,n} : i \in [n], n \ge 1\}$ are independent such that $|B_{i,n}| \le K$ and $\E[B_{i,n}] = 0$ for all $i \in [n]$ and all $n \ge 1$; and that $\{f_{i,n}: i \in [n], n \ge 1\}$ are uniformly bounded and uniformly equicontinuous functions either defined on a compact interval. Then, in probability,
\begin{align*}
\Big|\frac1n \sum_{i = 1}^n B_{i,n} f_{i,n}\Big|_\infty 
\xrightarrow{n\to\infty} 0.
\end{align*}
\end{lem}

\begin{proof}
The arguments are quite similar to those supporting \lemref{GC}.
Define 
\begin{align*}
\label{Sn}
S_n := \frac1n \sum_{i = 1}^n B_{i,n} f_{i,n}.
\end{align*}
Suppose without loss of generality that the functions are defined on the unit interval. Because they are uniformly equicontinuous, for any given $\eps > 0$, there is $\delta > 0$ such that $\sup_n \sup_i |f_{i,n}(t) - f_{i,n}(s)| \le \eps$ for all $s, t \in [0,1]$ such that $|t-s| \le \delta$. With $\eps > 0$ fixed, and $\delta$ as such, let $0 = t_1 < t_2 < \cdots < t_{m-1} < t_m = 1$ be such that $t_{j+1} - t_j \le \delta$ for all $j$.
Then for any $t \in [0,1]$, if $j$ is such that $t_j \le t \le t_{j+1}$,
\begin{align*}
|S_n(t) - S_n(t_j)| 
\le \frac1n \sum_{i = 1}^n |B_{i,n}| |f_{i,n}(t) - f_{i,n}(t_j)|
\le K \sup_n \sup_i |f_{i,n}(t) - f_{i,n}(t_j)|
\le K \eps.
\end{align*}
In particular,
\begin{align*}
|S_n|_\infty \le \max_j |S_n(t_j)| + K \eps.
\end{align*}
Furthermore, by Chebyshev's inequality, in probability,
\begin{align*}
\max_j |S_n(t_j)| \xrightarrow{n\to\infty} 0.
\end{align*}
Hence, in probability,
\begin{align*}
\limsup_n |S_n|_\infty \le K \eps.
\end{align*}
Since $\eps>0$ was chosen arbitrary, the proof is complete.
\end{proof}

The following is a well-known error bound for Riemann sums.
\begin{lem}
\label{lem:riemann}
Assume that $f:[a,b] \to \bbR$ is continuous with modulus of continuity $\omega$. Then 
\begin{align*}
\left|\int_a^b f(u) \d u - \frac1m \sum_{j=1}^m f(a + j (b-a)/m)\right|
&\le (b-a) \omega((b-a)/m) \\
&\le \frac{(b-a)^2}{m}\, |f'|_\infty \quad \text{if $f$ is Lipschitz.}
\end{align*}
\end{lem}

\begin{proof}
This well-known result is a simple consequence of partitioning $[a,b]$ into sub-intervals of length $(b-a)/m$.
\end{proof}

The next two lemmas are refinements of \lemref{GC} and \lemref{simple_EP}. They clearly subsume them, but they are also much deeper, and we only provide proof sketches, relying on arguments borrowed from \citep{van2000asymptotic}.

\begin{lem}
\label{lem:KS}
In the context of \lemref{GC}, for some constant $C$,
\begin{align*}
\E\Big[\sup_t |\hat F_{[n]}(t) - F_{[n]}(t)|\Big] 
\le C/\sqrt{n}.
\end{align*}
\end{lem}

\begin{proof}
The result is classical when the variables are not only independent, but also identically distributed, say $Y_{i,n} \sim F$ for all $i$ and all $n$, and is a special case of so-called entropy bounds on the supremum of an empirical processes of the form
\begin{equation}
S_n := \sup_{g \in \cG} \frac1{\sqrt{n}} \sum_{i=1}^n (g(Y_{i,n}) - \E[g(Y_{i,n})]).
\end{equation}
For example, assuming that the class $\cG$ is uniformly bounded, Cor~19.35 in \citep{van2000asymptotic} gives
\begin{equation*}
\E[S_n] \le C_0 J(\cG, F),
\end{equation*}
where $C_0$ is a constant and 
\begin{equation*}
J(\cG, F) := \int_0^\infty \sqrt{\log N(\eps, \cG, L^2(F))} \d \eps,
\end{equation*}
$N(\eps, \cG, L^2(F))$ denoting the $\eps$-bracketing number of the class $\cG$ with respect to the $L^2(F)$ metric.\footnote{Two functions $g_1, g_2$ such that $g_1 \le g_2$ pointwise define a bracket made of all functions $g$ such that $g_1 \le g \le g_2$. It is said to be an $\eps$-bracket with respect to $L^2(F)$, for a positive measure $F$, if $\int (g_2 - g_1)^2 \d F \le \eps^2$. 
Given a class of functions $\cG$, its $\eps$-bracketing number with respect to $L^2(F)$ is the minimum number of $\eps$-brackets needed to cover $\cG$.}

The proof of that result takes several pages, but a close examination reveals that the `identically distributed' property is not used in an essential way. Indeed, the assumption that the variables are iid is only used when applying Bernstein's concentration inequality (Lem~19.32 there), and it is well-known that the result applies in a generalized form to variables that are only independent, say $Y_{i,n} \sim F_{i,n}$. Everything follows from that, essentially verbatim, and yields
\begin{equation*}
\E[S_n] \le C_0 J(\cG, F_{[n]}).
\end{equation*}
It turns out that, for a given distribution function $F$, $J(\cG, F)$ can be bounded based on the modulus of continuity of $F$ (see Ex~19.6 in the same reference). 
And if $\omega$ is the modulus of continuity of $\{F_{i,n}\}$, then it is also a modulus of continuity for $F_{[n]}$, and with this we can bound $J(\cG, F_{[n]})$ independently of $n$ just based on $\omega$.

When dealing with the empirical distribution function, which is our focus here, the class is  taken to be $\cG := \{\IND{y \le t} : t \in \bbR\}$. For that class, $J(\cG, F) < \infty$ for any distribution function, and this implies via the arguments above that $\sup_n J(\cG, F_{[n]}) < \infty$, concluding the proof.
\end{proof}

\begin{lem}
\label{lem:simple_EP_bound}
Suppose that $\{B_i : i \ge 1\}$ are independent random variables that are centered and bounded in absolute value by $K$. And let $\{f_i : i \ge 1\}$ be $L$-Lipschitz functions on $[-t_0,t_0]$ with $f_i(0) = 0$ for all $i$. Then there is a some constant $C$ such that, for any $n \ge 1$,
\begin{align*}
\E\bigg[\sup_{|t| \le t_0} \Big|\frac1n \sum_{i = 1}^n B_{i} f_{i}(t)\Big|\bigg] 
\le \frac{C t_0}{\sqrt{n}}.
\end{align*}
\end{lem}

\begin{proof}
As in the proof of \lemref{KS}, we rely on entropy bounds. Here we use Dudley's entropy bound as presented in \cite[Th 2.3.6]{gine2016mathematical}. For a given $n$, let $S(t) := \frac1n \sum_{i = 1}^n B_{i} f_{i}(t)$, we have
\begin{equation}
S(t) - S(s) 
= \sum_{i = 1}^n \frac{B_i}n (f_{i}(t) - f_i(s)),
\end{equation}
with the variables $\frac{B_i}n (f_{i}(t) - f_i(s))$ being independent, centered, and bounded in absolute value by $(K/n) L |t-s|$. In \citep{gine2016mathematical}, by Eq~(3.8) and based on Def~2.3.5, the process $S(t)$ is sub-Gaussian on $[-t_0, t_0]$ with respect to the metric $\d(s,t) := (KL/\sqrt{n}) |s-t|$. Because $S(0) = 0$, Th~2.3.6 there gives that
\begin{align*}
\E\Big[\sup_{|t| \le t_0} |S(t)|\Big]
\le 4 \sqrt{2} \int_0^{D/2} \sqrt{\log(2 N(\eps))} \d \eps,
\end{align*}
where $D$ and $N(\eps)$ are the diameter and $\eps$-covering number of $[-t_0, t_0]$ with respect to $\d$. Immediately, $D = (KL/\sqrt{n}) (2t_0) = 2 KL t_0/\sqrt{n}$, and $N(\eps) \asymp (KL/\sqrt{n}) (t_0/\eps) \asymp D/\eps$. With a simple change of variable in the integral, this gives us
\begin{align*}
\E\Big[\sup_{|t| \le t_0} |S(t)|\Big]
\le C_1 KL t_0/\sqrt{n},
\end{align*}
for a universal constant $C_1$.
\end{proof}

\subsection{Proof of \lemref{uniform_convergence}}

We assume without loss of generality that $\theta^* = 0$.
We have 
\begin{align*}
R_i 
= \sum_{j=1}^n \IND{Y_j \le Y_i}
= n \hat\Psi_n(Y_i), \quad \text{where } \hat\Psi_n(y) := \frac1n \sum_{j=1}^n \IND{Y_j \le y}.
\end{align*}
Note that $\hat\Psi_n$ is the empirical distribution function of $Y_1, \dots, Y_n$. Although these are not iid, they are independent, and the Glivenko--Cantelli theorem applies to give that, in probability,
\begin{align*}
\sup_{y \in \bbR} |\hat\Psi_n(y) - \E[\hat\Psi_n(y)]| 
\xrightarrow{n\to\infty} 0.
\end{align*}
See \lemref{GC} for details.
Further, we have
\begin{align*}
\E[\hat\Psi_n(y)]
&= \frac1n \sum_{j=1}^n \P(Y_j \le y) \\
&= \frac1n \sum_{j=1}^n \P(Z_j \le y - f(x_j)) \\
&= \frac1n \sum_{j=1}^n \Phi(y-f(j/n)) \\
&\xrightarrow{n\to\infty} \Psi(y) := \int_0^1 \Phi(y-f(x)) \d x,
\end{align*}
where the convergence is by definition of the Riemann integral defining the limit. The convergence is in fact uniform in $y$. This comes from an application of \lemref{riemann} using with the fact that $x \mapsto \Phi(y-f(x))$ has derivative $f'(x) \phi(y-f(x))$, which has supremum norm bounded by $|f'|_\infty |\phi|_\infty < \infty$ (independent of $y$).
Hence, a simple application of the triangle inequality gives that, in probability,
\begin{align}
A_{1,n} := \sup_{y \in \bbR} |\hat\Psi_n(y) - \Psi(y)| 
\xrightarrow{n\to\infty} 0.
\end{align}
This is useful to us because $R_i = n \Psi(Y_i) \pm n A_{1,n}$, which then triggers
\begin{align*}
\widehat M_n(\theta)
= \frac1n \sum_{i=1}^n \Psi(Y_i) f(x_i-\theta) \pm A_{2,n},
\end{align*}
with $A_{2,n} := A_{1,n} |f|_\infty = o_P(1)$.
We may thus focus on the first term on the right-hand side.
We have
\begin{align*}
\frac1n \sum_{i=1}^n \Psi(Y_i) f(x_i-\theta) 
= \frac1n \sum_{i=1}^n \E[\Psi(Y_i)] f(x_i-\theta) + Q_n(\theta),
\end{align*}
with
\begin{align*}
Q_n(\theta) := \frac1n \sum_{i=1}^n (\Psi(Y_i)-\E[\Psi(Y_i)]) f(x_i-\theta).
\end{align*}
On the one hand, by a standard argument consisting in discretizing the values of $\theta$ and using the uniform continuity of $f$, we obtain
\begin{align*}
A_{3,n} := \sup_{\theta \in \bbR} |Q_n(\theta)| \xrightarrow{n\to\infty} 0,
\end{align*}
in probability. See \lemref{simple_EP} for details.
On the other hand, 
\begin{align*}
\frac1n \sum_{i=1}^n \E[\Psi(Y_i)] f(x_i-\theta)
&= \frac1n \sum_{i=1}^n \int_{-\infty}^\infty \Psi(z + f(i/n)) \phi(z) \d z \cdot f(i/n-\theta) \\
&\xrightarrow{n\to\infty} \int_0^1 \int_{-\infty}^\infty \Psi(z + f(x)) \phi(z) \d z \cdot f(x-\theta) \d x \\
&\qquad\qquad = \int_0^1 \int_0^1 \Phi_2(f(x)-f(t)) f(x-\theta) \d t \d x
= M(\theta),
\end{align*}
using again the definition of Riemann integral.
In fact the convergence is uniform in $\theta$, by an application of \lemref{riemann} to the function 
\begin{equation*}
g_\theta(x) := \int_{-\infty}^\infty \Psi(z + f(x)) \phi(z) \d z \cdot f(x-\theta),
\end{equation*}
whose derivative can be bounded independently of $\theta$ as follows:
\begin{align*}
|g_\theta'(x)|
&= \Big|\int_{-\infty}^\infty f'(x) \Psi'(z+f(x)) \phi(z) \d z \cdot f(x-\theta) + \int_{-\infty}^\infty \Psi(z + f(x)) \phi(z) \d z \cdot f'(x-\theta)\Big| \\
&\le |f'|_\infty |\Psi'|_\infty |f|_\infty + |\Psi|_\infty |f'|_\infty \\
&\le |f'|_\infty |\phi|_\infty |f|_\infty + |f'|_\infty,
\end{align*}
using the fact that $\phi$ is a density, that $\Psi$ is a distribution function, and that $\Psi'(y) = \int_0^1 \phi(y-f(x)) \d x$ is non-negative and bounded by $|\phi|_\infty$.
All combined, we can conclude that $\widehat M_n(\theta)$ indeed converges in probability as $n\to\infty$ to $M(\theta)$ uniformly in $\theta$.

\subsection{Proof of \prpref{local}}

Assume without loss of generality that $\theta^* = 0$.
Define
\begin{align*}
g(y) := \int_0^1 \Phi_2(y-f(x)) \d x,
\end{align*}
so that
\begin{align*}
M(\theta) = \int_0^1 g(f(t)) f(t-\theta) \d t.
\end{align*}
Note that $0\le g(y) \le 1$ for all $y$. 
When $f$ is Lipschitz, it is absolutely continuous with bounded derivative, so that by dominated convergence, $M$ is differentiable with derivative
\begin{align*}
M'(\theta) 
= - \int_0^1 g(f(t)) f'(t-\theta) \d t
= - \int_0^1 g(f(t+\theta)) f'(t) \d t.
\end{align*}
The reason we transferred $\theta$ to $g$ is to be able to differentiate again. Indeed, $g$ is also differentiable by dominated convergence, with derivative (recall that $\Phi' = \phi$)
\begin{align*}
g'(y) = \int_0^1 \int_{-\infty}^\infty \phi(z+y-f(x)) \phi(z) \d z \d x,
\end{align*}
which is bounded, so that $M'$ in its second form is also differentiable (by dominated convergence again), with derivative
\begin{align*}
M''(\theta)
= - \int_0^1 f'(t+\theta) g'(f(t+\theta)) f'(t) \d t.
\end{align*}
Therefore, $M''$ is twice differentiable (with bounded second derivative at that).

We now look at $\theta = 0$.
Let $G$ be the indeterminate integral of $g$.
We have
\begin{align*}
M'(0)
&= - \int_0^1 g(f(t)) f'(t) \d t \\
&= - [G(f(1)) - G(f(0))] 
= 0,
\end{align*}
by the fact that $f$ is 1-periodic.
We also have
\begin{align*}
M''(0)
&= - \int_0^1 g'(f(t)) f'(t)^2 \d t
\le 0,
\end{align*}
by the fact that $g'$ is non-negative (by simply looking at the integrand defining it, recalling that $\phi$ is a density). In fact the inequality is strict by our assumption on $\phi$, as it forces $g'$ to be strictly positive everywhere.

\subsection{Proof of \thmref{rate}}

When working with the raw responses $Y_1, \dots, Y_n$, the result can be proved using \cite[Th 5.52]{van2000asymptotic}. Since we work with the ranks $R_1,\dots,R_n$ instead, we elaborate, even though the core arguments are essentially the same. Assume without loss of generality that $\theta^* = 0$, so that we need to show that $\sqrt{n} \hat\theta_n$ is bounded in probability.


On the one hand, by definition, we have $\widehat M_n(\hat\theta_n) - \widehat M_n(0) \ge 0$. 
On the other hand, by consistency (since \thmref{consistent} applies), we have that $|\hat\theta_n|$ is small, and by the fact that $M$ is close to quadratic in the neighborhood of $0$ (\prpref{local}), we have that $M(\hat\theta_n)-M(0) \le -C_1 \hat\theta_n^2$ for some constant $C_1>0$.
Combined, these two observations yield
\begin{align*}
C_1 \hat\theta_n^2
\le \widehat M_n(\hat\theta_n) - M(\hat\theta_n) - (\widehat M_n(0) - M(0)),
\end{align*}
with probability tending to 1.

Let $f_i(\theta) := f(x_i -\theta) - f(x_i)$.
For any $\theta$, we have
\begin{align}
& \widehat M_n(\theta) - M(\theta) - (\widehat M_n(0) - M(0)) \\
&= \frac1n \sum_i \big(\hat\Psi_n(Y_i) - \E[\hat\Psi_n](Y_i)\big) f_i(\theta) \label{term1} \\
&\quad + \frac1n \sum_i \big(\E[\hat\Psi_n](Y_i) - \Psi(Y_i)\big) f_i(\theta) \label{term2} \\
&\quad + \frac1n \sum_i \big(\Psi(Y_i) - \E[\Psi(Y_i)]\big) f_i(\theta) \label{term3} \\
&\quad + \frac1n \sum_i \E[\hat\Psi(Y_i)] f_i(\theta) - (M(\theta) - M(0)). \label{term4}
\end{align}
We saw in the proof of \lemref{uniform_convergence} that the terms in \eqref{term2} and \eqref{term4} are Riemannian sums and at most of order  $O(1/n)$ uniformly in $\theta$ given that the $|f_i|_\infty \le 2 |f|_\infty$. (This is crude, but enough for our purposes here.)
For \eqref{term1}, we apply \lemref{KS} together with Markov's inequality to get that, in probability as $n\to\infty$, 
\begin{equation*}
\sup_{y \in \bbR} |\hat\Psi_n(y) - \E[\hat\Psi_n](y)| \le C_2/\sqrt{n}.
\end{equation*}
With this, and the fact that $|f_i(\theta)| \le |f'|_\infty |\theta|$, we have that the quantity in \eqref{term1} is bounded in absolute value by $(C_2/\sqrt{n}) |f'|_\infty |\theta|$ for all $\theta$, that is, this term is $O(|\theta|/\sqrt{n})$ uniformly in $\theta$.
Hence, if $S_n(\theta)$ denotes the term in \eqref{term3}, we have with probability tending to~1,
\begin{align*}
C_1 \hat\theta_n^2 
\le S_n(\hat\theta_n) + C_3 (|\hat\theta_n|/\sqrt{n} + 1/n).
\end{align*}

Let $C_4 > 0$ be such that $C_1 \theta^2 - C_3 (|\theta|/\sqrt{n} + 1/n) \ge \theta^2/C_4$ whenever $|\theta| \ge C_4/\sqrt{n}$, so that $S_n(\hat\theta_n) \ge \hat\theta_n^2/C_4$ when $\sqrt{n} |\hat\theta_n| \ge C_4$.
Let $J_0$ be the smallest integer such that $2^{J_0} \ge C_4$.
Then, for $J \ge J_0$, we have
\begin{align*}
\P(\sqrt{n} |\hat\theta_n| \ge 2^J)
&= \sum_{j \ge J} \P(2^j < \sqrt{n} |\hat\theta_n| \le 2^{j+1}) \\
&\le \sum_{j \ge J} \P\Big(\max_{\sqrt{n} |\theta| \le 2^{j+1}} S_n(\theta) \ge (2^j/\sqrt{n})^2/C_4\Big) \\
&\le \sum_{j \ge J} \frac{(C_5/\sqrt{n}) (2^{j+1}/\sqrt{n})}{(2^j/\sqrt{n})^2/C_4} \\
&= C_6 2^{-J} \xrightarrow{J\to\infty} 0,
\end{align*}
where $C_5$ is the constant of \lemref{simple_EP_bound}, and we used that lemma and Markov's inequality in the corresponding line.
We can thus conclude that $\sqrt{n} \hat\theta_n$ is bounded in probability.

\subsection{Proof of \thmref{normal}}

We assume without loss of generality that $\theta^* = 0$.
The derivation of the limiting distribution of the R-estimator follows via an application of the so-called argmax theorem. This standard route is described, for example, in \cite[Sec 5.9]{van2000asymptotic}. 
It goes like this. By a simple change of variables and by definition of $\hat\theta_n$, $h_n := \sqrt{n} \hat\theta_n$ maximizes 
\begin{equation*}
W_n(h) := r_n \big[\widehat M_n(h/\sqrt{n}) - \widehat M_n(0)\big].
\end{equation*}
This is true for any $r_n > 0$ and, with probability tending to one according to \thmref{rate}, it is true even if $W_n$ is restricted to $[-a_n, a_n]$ for any given sequence $a_n\to\infty$. Suppose there is a choice of $r_n$ that leads to the weak convergence of $W_n$ to $W$ in some appropriate sense, where $W$ has a unique maximizer. Then it is reasonable to anticipate that $h_n$ will converge to that maximizer in some way. This is indeed the case under some mild assumptions.
The following is a special case of \cite[Cor 5.58]{van2000asymptotic}.
\begin{lem}
\label{lem:argmax}
Suppose that a sequence of processes $W_{n}$ defined on $[-a_n,a_n]$ for some sequence $a_n \to \infty$ converges weakly in the uniform topology on every fixed compact interval to a process $W$ with continuous sample paths each having a unique maximum point $h^*$ (almost surely). If $h_n$ maximizes $W_n$, and $(h_n)$ is uniformly tight, then $h_n$ converges weakly to $h^*$.
\end{lem}

Back to our situation, we have established the tightness of $\{h_n\}$ in \thmref{rate}. It therefore remains to show that $W_n$ converges weakly to an appropriate stochastic process for a proper choice of $r_n$. We will see that $r_n := n$ is the correct choice (up to an arbitrary multiplicative factor) and that the limit process is a simple Gaussian process. In what follows, we let $a_n \to \infty$ slowly (e.g., $a_n = \log n$).

So far, we have worked with the ranks using rather elementary means, but now we turn to more sophisticated tools. Specifically, we use the projection method of H\'ajek. The following is a special case of \cite[Th~4.2]{hajek1968asymptotic} with some minor modifications. 

\begin{lem}
\label{lem:hajek}
Suppose $Y_1, \dots, Y_n$ are independent with respective distribution functions $F_1, \dots, F_n$. Define $M = \sum_i b_i R_i/n$, where $R_1, \dots, R_n$ are the respective ranks of $Y_1, \dots, Y_n$, and $b_1, \dots, b_n$ are reals. Then, for a universal constant $C$,
\begin{align*}
\E\Big[\Big(M - \mu - \sum_{i=1}^n V_i\Big)^2\Big] 
\le \frac{C}n \sum_{i=1}^n b_i^2, 
\end{align*}
where 
\begin{align*}
\mu := \sum_{i=1}^n b_i \frac1n \sum_{j=1}^n \int F_j(t) \d F_i(t),
&&
V_i := \frac1n \sum_{i=1}^n (b_j-b_i) \int \big[\IND{Y_i\le t} - F_i(t)\big] \d F_j(t).
\end{align*}
\end{lem}

This result thus provides an approximation of a linear combination of ranks (which are dependent) by a linear combination of independent random variables, and the latter is essentially ready for an application of a central limit theorem.
We apply it to
\begin{align*}
W_n(h)
= \sum_{i=1}^n f_i(h) \frac{R_i}n, \quad f_i(h) := f(x_i - h/\sqrt{n}) - f(x_i).
\end{align*}
Note that $f_i$ depends on $n$ and recall that $x_i = i/n$. 
In \lemref{hajek}, $b_i$ corresponds here to $f_i(h)$ and $F_i$ to $\Phi(\cdot - f(x_i))$.
Hence, $\mu$ in the lemma is given by
\begin{align*}
&\sum_{i=1}^n f_i(h) \frac1n \sum_{j=1}^n \int_{-\infty}^\infty \Phi(t - f(x_j)) \phi(t - f(x_i)) \d t \\
&= \sum_{i=1}^n f_i(h) \frac1n \sum_{j=1}^n \Phi_2(f(x_i) - f(x_j)) \\
&= \sum_{i=1}^n f_i(h) \bigg[\int_0^1 \Phi_2(f(x_i) - f(x)) \d x \pm \frac{|f'|_\infty}{2n} \bigg] \\
&= \sum_{i=1}^n (f(x_i-h)-f(x_i)) \int_0^1 \Phi_2(f(x_i) - f(x)) \d x \pm n \frac{|f'|_\infty a_n}{\sqrt{n}} \frac{|f'|_\infty}{n}.
\end{align*}
In the 3rd equality, we used \lemref{riemann} and the fact that $x \mapsto \Phi_2(f(x_i) - f(x))$ has derivative $-f'(x) \phi_2(f(x_i) - f(x))$, whose supnorm is bounded by $|f'|_\infty$.
Defining 
\begin{align*}
g_h(t) 
:= (f(t-h)-f(t)) \int_0^1 \Phi_2(f(t) - f(x)) \d x, 
\end{align*}
we have 
\begin{align*}
&\sum_{i=1}^n (f(x_i-h)-f(x_i)) \int_0^1 \Phi_2(f(x_i) - f(x)) \d x \\
&= \sum_{i=1}^n g_h(x_i) \\
&= n\, \bigg[\int_0^1 g_h(t) \d t \pm \frac{|g_h'|_\infty}{n} \bigg] \\
&= n \big[M(h/\sqrt{n}) - M(0)\big] \pm \omega_1(a_n/\sqrt{n}),
\end{align*}
using \lemref{riemann} in the 3rd equality, and where $\omega_1(\eps) := \sup_t \sup_{|s| \le \eps} |f'(t+s) - f'(t)|$, which is the modulus of continuity of $f'$. Note that $\omega_1(\eps) \to 0$ when $\eps \to 0$ by the fact that $f$ is assumed to be continuously differentiable (and 1-periodic).
Therefore, $\mu$ in the lemma is equal to 
\begin{align*}
n \big[M(h/\sqrt{n}) - M(0)\big] \pm \omega_1(a_n/\sqrt{n}) \pm \frac{|f'|_\infty^2 a_n}{\sqrt{n}},
\end{align*}
with the remainder terms tending to 0 and, for $h$ fixed, 
\begin{align*}
n \big[M(h/\sqrt{n}) - M(0)\big]
\xrightarrow{n\to\infty} \tfrac12 M''(0) h^2.
\end{align*}
With the fact that $M''$ is continuous under our assumption that $f$ is continuously differentiable, we conclude that $\mu$ is equal to
\begin{align*}
\tfrac12 M''(0) h^2 \pm Q_{1,n}, \quad Q_{1,n} \xrightarrow{n\to\infty} 0 \text{ in probability}.
\end{align*}
As for $V_i$ in the lemma, it is equal to
\begin{align}
&\frac1n \sum_{j=1}^n (f_j(h)-f_i(h)) \int \big[\IND{Y_i\le t} - \Phi(t-f(x_i))\big] \phi(t-f(x_j)) \d t \\
&= \frac1n \sum_{j=1}^n (f_j(h)-f_i(h)) \int \big[\IND{Z_i \le z + f(x_j)-f(x_i)} - \Phi(z + f(x_j)-f(x_i))\big] \phi(z) \d z \\
&\stackrel{d}{=} \frac1n \sum_{j=1}^n (f_j(h)-f_i(h)) \big[\Phi(Z_i + f(x_j)-f(x_i)) - \Phi_2(f(x_j)-f(x_i))\big], \label{normal_proof1}
\end{align}
using the fact that $\phi$ is symmetric about 0.
Note that
\begin{align*}
f_j(h) - f_i(h)
&= f(x_j+h/\sqrt{n})-f(x_j)-\big[f(x_i+h/\sqrt{n})-f(x_i)\big] \\
&= (h/\sqrt{n}) \big[f'(x_j)-f'(x_i) \pm \omega_1(h/\sqrt{n})\big] \\
&= (h/\sqrt{n}) [f'(x_j)-f'(x_i)] \pm (a_n/\sqrt{n}) \omega_1(a_n/\sqrt{n}).
\end{align*}
Hence, the quantity in \eqref{normal_proof1} is equal to
\begin{align*}
&\frac1n \sum_{j=1}^n (h/\sqrt{n}) (f'(x_j)-f'(x_i)) \big[\Phi(Z_i + f(x_j)-f(x_i)) - \Phi_2(f(x_j)-f(x_i))\big] \pm (a_n/\sqrt{n}) \omega_1(a_n/\sqrt{n}) \\
&= \frac{h}{\sqrt{n}} \big[\Lambda(x_i, Z_i) \pm \omega_1(1/n)\big] \pm (a_n/\sqrt{n}) \omega_1(a_n/\sqrt{n}),
\end{align*}
where
\begin{align*}
\Lambda(t, z)
&:= \int_0^1 (f'(x)-f'(t)) \big[\Phi(z + f(x)-f(t)) - \Phi_2(f(x)-f(t))\big] \d x \\
&:= \int_0^1 (f'(x)-f'(t)) \Xi(f(x)-f(t), z) \d x,
\end{align*}
using \lemref{riemann}.
We thus conclude that $\sum_{i=1}^n V_i$ is equal, in distribution, to
\begin{align*}
\frac{h}{\sqrt{n}} \sum_{i=1}^n \Lambda(x_i, Z_i) \pm Q_{2,n}, \quad Q_{2,n} \xrightarrow{n\to\infty} 0 \text{ in probability}. 
\end{align*}

By Markov's inequality and the fact that we only consider $W_n(h)$ for $|h| \le a_n$, we thus have that
\begin{align*}
W_n(h)
&\stackrel{d}{=} \tfrac12 M''(0) h^2 \pm Q_{1,n} + \frac{h}{\sqrt{n}} \sum_{i=1}^n \Lambda(x_i, Z_i) \pm Q_{2,n} \pm Q_{3,n},
\end{align*}
where 
\begin{align*}
\E[Q_{3,n}^2] 
&\le \frac{C}n \sum_{i=1}^n f_i(h)^2 
\le C (|f'|_\infty h/\sqrt{n})^2
\le \frac{C |f'|_\infty^2 a_n^2}n,
\end{align*}
so that $Q_{3,n} \to 0$ as $n\to 0$ in probability.
More succinctly, therefore, 
\begin{align}
W_n(h)
&\stackrel{d}{=} \tfrac12 M''(0) h^2 + \frac{h}{\sqrt{n}} \sum_{i=1}^n \Lambda(x_i, Z_i) + o_P(1). \label{W_linear}
\end{align}
Hence, by Slutsky's theorem in the form of \cite[Th~7.15]{kosorok2008introduction}, it suffices to look at 
\begin{equation}
G_n(h) := \tfrac12 M''(0) h^2 + \frac{h}{\sqrt{n}} \sum_{i=1}^n \Lambda(x_i, Z_i),
\end{equation}
which is an exceedingly simple process. (This is because we are effectively in a classical setting, even though the ranks obfuscate that.)

Indeed, note that $\Lambda(x_i, Z_i)$ is centered and bounded in absolute value by $2 |f'|_\infty$, and
\begin{align*}
&\frac1n \sum_{i=1}^n \Var[\Lambda(x_i, Z_i)] \\
&=\frac1n \sum_{i=1}^n \E[\Lambda(x_i, Z_i)^2] \\
&= \frac1n \sum_{i=1}^n \int_{-\infty}^\infty \bigg[\int_0^1 (f'(x) - f'(x_i))\, \Xi(f(x)-f(x_i), z) \d x\bigg]^2 \phi(z) \d z \\
&\xrightarrow{n\to\infty} \gamma^2,
\end{align*}
again applying \lemref{riemann}. 
Therefore, by Lyapunov's central limit theorem, 
\begin{align*}
\frac{1}{\sqrt{n}} \sum_{i=1}^n \Lambda(x_i, Z_i)
\overset{n\to\infty}{\Longrightarrow} \cN(0, \gamma^2).
\end{align*}
From this, it follows that $G_n$ converges weakly on every bounded interval to the Gaussian process given by
\begin{align*}
G(h) := \tfrac12 M''(0) h^2 + h \gamma U,
\end{align*}
where $U$ is a standard normal random variable. The limit process could not be simpler. In particular, $G$ has continuous sample paths (in fact, its sample paths are parabolas), and recalling that $M''(0) < 0$ by \prpref{local}, it is clear that $G$ has a unique maximum point at $h^* := (\gamma/M''(0)) U$. Note that $h^*$ is normal with mean zero and variance $\gamma^2/M''(0)^2$.
The proof of the theorem then follows from an application of \lemref{argmax}.

\subsection*{Acknowledgments}
We are very grateful to Richard Nickl and Nicolas Verzelen for helpful discussions and pointers.   

\small
\bibliographystyle{chicago}
\bibliography{ref}

\end{document}